\documentclass[12pt,leqno,letterpaper,fleqn]{amsart}

\usepackage{amsmath,amstext,amsthm,amssymb,amsxtra,bbm}
\usepackage[top=1.25in, bottom=1.25in, left=1.25in, right=1.25in]{geometry}

\usepackage{mathtools}
\mathtoolsset{showonlyrefs,showmanualtags}

\usepackage{hyperref}
\hypersetup{
    pdfstartview={FitH},    
    colorlinks=true,       
    linkcolor=blue,          
    citecolor=magenta,        
    filecolor=magenta,      
    urlcolor=cyan           
}

\usepackage[msc-links]{amsrefs}

\theoremstyle{plain} 

\newtheorem{proposition}[equation]{Proposition}
\newtheorem{theorem}[equation]{Theorem}
\newtheorem{corollary}[equation]{Corollary}
\newtheorem{conjecture}[equation]{Conjecture}

\numberwithin{equation}{section}

\title[Improving estimates for discrete polynomial averages]{Improving estimates for discrete polynomial averages}

\author[R. Han]{Rui Han}
\address{Rui Han, School of Mathematics, Georgia Institute of Technology, Atlanta GA 30332, USA}
\email{rui.han@math.gatech.edu}
\thanks{R. H. is supported in part by National Science Foundation grant DMS-1800689.}

\author[V. Kova\v{c}]{Vjekoslav Kova\v{c}}
\address{Vjekoslav Kova\v{c}, Department of Mathematics, Faculty of Science, University of Zagreb, Bijeni\v{c}ka cesta 30, 10000 Zagreb, Croatia}
\email{vjekovac@math.hr}
\thanks{V. K. is supported in part by the Croatian Science Foundation under project UIP-2017-05-4129 (MUNHANAP), and by the Fulbright Scholar Program. He also appreciates hospitality of the Georgia Institute of Technology in the academic year 2019--2020.}

\author[M. T. Lacey]{Michael T. Lacey}
\address{Michael T. Lacey, School of Mathematics, Georgia Institute of Technology, Atlanta GA 30332, USA}
\email{lacey@math.gatech.edu}
\thanks{M. T. L. is supported in part by National Science Foundation grant DMS-1600693, and by Australian Research Council grant DP160100153.}

\author[J. Madrid]{Jos\'{e} Madrid}
\address{Jos{\'e} Madrid, Department of Mathematics, University of California Los Angeles, Los Angeles CA  90095-1555}
\email{jmadrid@math.ucla.edu}

\author[F. Yang]{Fan Yang}
\address{Fan Yang, School of Mathematics, Georgia Institute of Technology, Atlanta GA 30332, USA}
\email{ffyangmath@gmail.com}
\thanks{F. Y. is supported in part by NSF DMS-1600693 and AMS-Simons Travel grant 2019--2021.}

\date{\today}

\subjclass[2020]{Primary 42A45; 
Secondary 42B15, 
11L07} 

\keywords{improving estimate, polynomial, discrete average, discrete fractional integral}

\begin{document}

\begin{abstract}
For a polynomial $P$ mapping the integers into the integers, define an averaging operator $A_{N} f(x):=\frac{1}{N}\sum_{k=1}^N f(x+P(k))$ acting on functions on the integers. We prove sufficient conditions for the $\ell^{p}$-improving inequality
\begin{equation*}
\|A_N f\|_{\ell^q(\mathbb{Z})} \lesssim_{P,p,q} N^{-d(\frac{1}{p}-\frac{1}{q})} \|f\|_{\ell^p(\mathbb{Z})}, \qquad  N \in\mathbb{N},
\end{equation*}
where $1\leq p \leq q \leq \infty$. For a range of quadratic polynomials, the inequalities established are sharp, up to the boundary of the allowed pairs of $(p,q)$. For degree three and higher, the inequalities are close to being sharp. In the quadratic case, we appeal to discrete fractional integrals as studied by Stein and Wainger. In the higher degree case, we appeal to the Vinogradov Mean Value Theorem, recently established by Bourgain, Demeter, and Guth.
\end{abstract}

\maketitle

\section{Introduction}
Discrete Radon averaging operators are our focus.
Let $ P$ be a polynomial of one variable mapping the integers to the integers and of degree $ d$.  Set an \emph{average} and a \emph{fractional integral} operator to be
\begin{align}  \label{eq:polyaver}
 A _{N} f(x)&:=\frac{1}{N}\sum_{k=1}^N f(x+P(k)) ,
 \\  \label{e:I}
 I _{\lambda } f (x) &:= \sum_{k=1} ^{\infty }  \frac {f (x+ P (k))} {k ^{\lambda }}, \qquad 0< \lambda < 1.
\end{align}
Throughout, functions $ f$ can be assumed to be finitely supported.
We write $A\lesssim B$ if there exists an absolute constant $C$ such that  $A\leq CB$. If the constant depends on parameters $\lambda,\mu,\ldots$  we denote that with a subscript, such as $A\lesssim_{\lambda,\mu,\ldots}B$.
The inequalities of interest are
\begin{align}\label{eq:polyest}
\|A_N f\|_{\ell^q(\mathbb{Z})} & \lesssim_{P,p,q} N^{-d(\frac{1}{p}-\frac{1}{q})} \|f\|_{\ell^p(\mathbb{Z})},   \qquad N \in \mathbb N ,   \  1<p < q < \infty .
\\  \label{e:Iest}
\| I _{\lambda } f\|_{\ell^q(\mathbb{Z})} & \lesssim_{P,\lambda,p,q}   \|f\|_{\ell^p(\mathbb{Z})}, \qquad 1 < p < q < \infty .
\end{align}
The inequality in \eqref{eq:polyest} should hold \emph{uniformly in $ N \in \mathbb N $}.
Using the notation $\lesssim_{P,p,q}$ we emphasize that the estimate has to be uniform in $N$ and $f$, but the implicit constant is allowed to depend on $P$, $p$, and $q$. One should similarly interpret the notation $\lesssim_{P,\lambda,p,q}$ in \eqref{e:Iest}. The exponent in \eqref{eq:polyest}, $-d(1/p-1/q)$,  is the best possible one, as is trivially seen by taking $f=\mathbbm{1}_{\{1,2,3,\ldots,2P(N)\}}$ and letting $N\to\infty$.

Set $ \mathcal A (P)$ to be the set of all $ (p,q)$ for which \eqref{eq:polyest} holds.  Set $ \mathcal I (P)$  to be the set of all $ (p,q, \lambda )$
for which \eqref{e:Iest} holds.
These two classes are related through

\begin{proposition}\label{p:IA} These two relations between $ \mathcal A (P)$ and $ \mathcal I (P)$ hold, where $ d$ is the degree of $ P$.
\begin{enumerate}\setlength\itemsep{.75em}
\item  If $ (p,q) \in \mathcal A (P)$, then $ (p,q, \lambda ) \in \mathcal I (P)$ for
$   0<  1-\lambda  < \min  \{1,  d (1/p - 1/q) \} $.
\item If $ (p,q, 1 - d (1/p - 1/q) ) \in \mathcal I (P)$, then $ (p,q) \in \mathcal A (P)$.
\end{enumerate}
\end{proposition}

We will define two closely related concepts in \S\ref{s:sparse}, and phrase some conjectures about them.

Concerning the collection $ \mathcal I (P)$, the relevant conjecture \cite{MR2820148}*{pg. 597} is

\begin{conjecture}\label{j:I}  The inequality \eqref{e:Iest} holds if and only if
$ 1- \lambda  \leq  d (1/p - 1/q)$, $ 1/q < \lambda $, and $ 1- \lambda < 1/p$.
\end{conjecture}

Discrete fractional integrals were studied by Stein and Wainger \cites{MR1945293,MR1771530}, with further contributions by
Oberlin \cite{MR1825254}, and Ionescu and Wainger \cite{MR2188130}.  In particular, the case of quadratic $ P (x) = x ^2 $ is
completely resolved.
Our first main theorem provides a sharp, up to the endpoint, bound for most quadratic polynomials.
Note that, when studying \eqref{eq:polyest}, interpolation with the trivial estimates for $q=p$ allows us to additionally assume $q=p'$ and $p<2$.

\begin{theorem}\label{t:2}
For a quadratic polynomial $P(x)=ax^2+bx+c$ with non-negative integer coefficients and $ N\in \mathbb N $,
the inequality \eqref{eq:polyest} holds in the range
\begin{equation*}
\big\{ (p,q) \,:\, {\textstyle\frac{1}{q}\leq\frac{1}{p}},\ {\textstyle\frac{2}{q}>\frac{1}{p}},\ {\textstyle\frac{1}{q}>\frac{2}{p}-1} \big\}.
\end{equation*}
More precisely,
for every $3/2< p\leq 2$ and $N\in\mathbb N$ one has
\begin{equation}\label{e:abc}
\|A_{N}f\|_{\ell^{p'}(\mathbb{Z})} \lesssim_p \left(2a+\frac{b}{N}\right)(2aN+b)^{-2(\frac{1}{p}-\frac{1}{p'})}\|f\|_{\ell^p(\mathbb{Z})}.
\end{equation}
\end{theorem}

It seems reasonable to conjecture that in \eqref{eq:polyest} the bounds, provided they hold, depend upon the polynomial only through its degree $ d$.  Right now, we do not know that this is true even in the quadratic case.

Pierce \cites{MR2820148,MR2872554} also studied the fractional integrals. In particular \cite{MR2820148} points to the relationship
to the (at that point unresolved) Vinogradov Mean Value Conjecture.
It reveals itself through the need for bounds on the
exponential sums
\begin{equation}\label{eq:expsums}
S_N(t_1,t_2,\ldots,t_d) := \frac{1}{N}\sum_{k=1}^{N} e^{2\pi i (k t_1 + k^2 t_2 + \cdots + k^d t_d)}; \quad t_1,t_2,\ldots,t_d\in\mathbb{T}=\mathbb{R}/\mathbb{Z}.
\end{equation}
This theme was further elaborated on by Kim \cite{MR3350107}.
Using the work of Bourgain, Guth, and Demeter \cite{MR3548534}, we establish

\begin{theorem}\label{thm:onedim}
Let $P$ be an arbitrary polynomial of degree $d\geq 3$ mapping the set of integers back into itself.
Averages \eqref{eq:polyaver} satisfy estimate \eqref{eq:polyest} for exponents $p,q$ in the triangular range
\begin{equation}\label{eq:polyrange}
\big\{ (p,q) \,:\, {\textstyle\frac{1}{q}\leq\frac{1}{p}},\ {\textstyle\frac{d^2+d+1}{q}>\frac{d^2+d-1}{p}},\ {\textstyle\frac{d^2+d-1}{q}>\frac{d^2+d+1}{p}-2} \big\}.
\end{equation}
Specializing $q=p'$ the range \eqref{eq:polyrange} reduces to
\begin{equation}\label{eq:polyrangespecial}
2-{\textstyle\frac{2}{d^2+d+1}}<p\leq2.
\end{equation}
\end{theorem}

We will regard averages \eqref{eq:polyaver} as ``projections'' of the following higher-dimensional polynomial averages.
Writing $ (x_1,x_2,\ldots,x_d)\in\mathbb{Z}^d$, consider
\begin{equation}\label{eq:highaver}
\widetilde{A}_N f(x_1,x_2,\ldots,x_d) := \frac{1}{N}\sum_{k=1}^{N} f(x_1+k,x_2+k^2,\ldots,x_d+k^d) .
\end{equation}
This time we want to prove $\ell^p$-improving estimates of the form
\begin{equation}\label{eq:highest}
\|\widetilde{A}_N f\|_{\ell^q(\mathbb{Z}^d)} \lesssim_{d,p,q} N^{-\frac{d(d+1)}{2}(\frac{1}{p}-\frac{1}{q})} \|f\|_{\ell^p(\mathbb{Z}^d)}.
\end{equation}
The exponent $-\frac{d(d+1)}{2}(\frac{1}{p}-\frac{1}{q})$ is the most improvement in $N$ that one can expect, as in easily seen by taking $f$ to be the indicator function of
$\{1,2,\ldots,2N\}\times\cdots\times\{1,2,\ldots,2N^d\}$.

\begin{theorem}\label{thm:highdim}
If $d\geq 3$, then averages \eqref{eq:highaver} satisfy estimate \eqref{eq:highest} for exponents $p,q$ in the same range \eqref{eq:polyrange}, as in Theorem~\ref{thm:onedim}.
\end{theorem}

As we explain in \S\ref{s:examples}, the bounds above are very close to optimal. It would be interesting to find the optimal open ranges of exponents $p$ and $q$ in Theorems~\ref{thm:onedim} and \ref{thm:highdim}.

Let us also remark that continuous-parameter results similar to Theorem~\ref{thm:highdim} have already been present in the literature for a while. For example, Christ \cite{MR1654767} has essentially settled all $\textup{L}^p$-improving properties of the convolution operator associated with the continuous moment curve $[-1,1]\to\mathbb{R}^d$, $t\mapsto(t,t^2,\ldots,t^d)$. Proofs of such results do not rely on number theory and sometimes not even on the Fourier analysis (as was the case with \cite{MR1654767}), but rather on geometrical considerations and enumerative combinatorics.

Our proof of Theorem~\ref{t:2} is done by using the essentially sharp results about discrete fractional integrals due to Stein and Wainger.
The argument can be reversed. We use our higher degree results on $\ell^{p}$-improving to deduce results about discrete fractional integrals.
To indicate the range of results that can be proved, we define
\begin{equation}\label{e:Id}
\widetilde{I} _{d, \lambda } f (x_1 ,\dotsc, x_d) := \sum_{k=1 } ^{\infty } k ^{- \lambda } f (x_1+k ,\dotsc, x_d+k^d),  \qquad  0 < \lambda <1
\end{equation}
for finitely supported  functions $f $.  The following corollary extends works by Kim \cite{MR3350107}*{Cor. 2.2, 2.3}, and is very nearly sharp.

\begin{corollary}\label{c:I}
Suppose that the indices $ (p,q)$ belong to the range \eqref{eq:polyrange}.
\begin{enumerate}\setlength\itemsep{.75em}
\item Let $P$ be an arbitrary polynomial of degree $d\geq 3$ mapping the set of integers back into itself. We have the inequality \eqref{e:Iest} provided that $ 0 < 1- \lambda < d (1/p - 1/q) $.
\item Take $d\geq 3$. We have the inequality $ \lVert \widetilde{I} _{d, \lambda } f\rVert _{\ell^{q}(\mathbb{Z}^d)} \lesssim _{d, \lambda, p, q } \lVert f\rVert_{\ell^{p}(\mathbb{Z}^d)} $ provided that $ 0 < 1- \lambda < (1/2)d(d+1) (1/p - 1/q) $.
\end{enumerate}
\end{corollary}

\smallskip
The study of improving estimates for averages has been studied for decades in the Euclidean setting.  It was recently recognized that
some of these inequalities can be further extended to so-called \emph{sparse bounds}.  The latter imply the strongest known weighted
estimates; see \cites{2017arXiv170208594L,2017arXiv170407810C,2019arXiv190600329H}.
The study of the discrete variants has a much shorter history.  Qualitative results were established in \cites{MR3892403,MR3933540}
for discrete singular  Radon transforms.
On the other hand, discrete spherical averages admit a robust variant of their continuous analogs
\cites{180409260H,180509925,180906468,181002240,180409845}.

Improving inequalities and sparse bounds are closely related, but the connection is far more delicate in the discrete case.
In particular, the sparse bounds  proved in \cites{2019arXiv190902883H,2019arXiv190705734H} for averages along the prime and square integers, respectively,  rely upon the Hardy--Littlewood Circle method. We do not know another way to prove those bounds; see the conjectures in \S\ref{s:sparse}.

\section{Relating averages to fractional integrals and vice versa}
\begin{proof}[Proof of Proposition~\ref{p:IA}] Assume that $ (p,q) \in \mathcal A (P)$, that is \eqref{eq:polyest} holds.
For $ 0<  1-\lambda  < \min  \{1,  d (1/p - 1/q) \} $ we have
\begin{align*}
\lVert I _{\lambda } f  \rVert _{\ell^q(\mathbb{Z})} & \lesssim_{\lambda} \sum_{j =1} ^{\infty }   2 ^{(1- \lambda ) j  }\lVert  A _{2 ^{j}} f  \rVert_{\ell^q(\mathbb{Z})}
\\
& \lesssim_{P,p,q} \lVert f\rVert_{\ell^p(\mathbb{Z})}   \sum_{j =1} ^{\infty }   2 ^{ ( 1- \lambda - d (1/p - 1/q) )j }  \lesssim_{\lambda} \lVert f\rVert_{\ell^p(\mathbb{Z})}.
\end{align*}

In the reverse direction,  observe that
$A_N f   \leq  N ^{-1+ \lambda } I _{\lambda } f$.
So if $ (p,q, 1 - d (1/p - 1/q) ) \in \mathcal I (P)$, the result follows.
\end{proof}

\begin{proof}[Proof of Corollary~\ref{c:I} assuming Theorems~\ref{thm:onedim} and \ref{thm:highdim}]
The first part is a direct consequence of Theorem~\ref{thm:onedim} and Proposition~\ref{p:IA}.
For the second part we again write
\begin{equation*}
\lVert \widetilde{I} _{d,\lambda } f  \rVert _{\ell^q(\mathbb{Z}^d)} \lesssim_{\lambda} \sum_{j =1}^{\infty }   2 ^{(1- \lambda ) j  }\lVert  \widetilde{A} _{2 ^{j}} f  \rVert_{\ell^q(\mathbb{Z}^d)}
\end{equation*}
and then apply Theorem~\ref{thm:highdim} in exactly the same way.
\end{proof}

\section{Quadratic polynomials}

\subsection*{Case $d=2$, $P(x)=x^2$}
For this particular choice of the polynomial, Conjecture~\ref{j:I} has been verified by Stein and Wainger  \cites{MR1771530,MR1945293}; also see \cite{MR2188130}*{Cor 1.3}.
Note that for $ 3/2 < p < 2$, $ q=p'$, and $ 1- \lambda = 2 (1/p-1/q)$, we have $ (p, q, \lambda ) \in \mathcal I (x^2)$,
so the result follows.

\subsection*{General quadratic polynomials}
Let us turn to the proof of \eqref{e:abc}.
In order to avoid confusion we will write the relevant polynomial $P$ in the superscript of $A_N^P$.

We define $g\colon\mathbb Z\to\mathbb{R}$ by
$$
g(4am)=f(m) \ \text{for all} \ m\in\mathbb Z \ \text{and} \ g(n)=0 \ \text{if}\ 4a\nmid n.
$$
Since $f$ is supported in $[-(aN^2+bN+\frac{b^2}{4a}),aN^2+bN+\frac{b^2}{4a}]$, we have that $g$ is supported in $[-(2aN+b)^2,(2aN+b)^2]$. Therefore
\begin{eqnarray*}
A_{N}^{P}f(x)&=&\frac{1}{N}\sum_{n\leq N}f(x+an^2+bn+c)\\
&=&\frac{1}{N}\sum_{n\leq N}g(4ax+4a^2n^2+4abn+4ac)\\
&=&\frac{1}{N}\sum_{n\leq N}g(4a(x+c)-b^2+(2an+b)^2)\\
&\leq&\frac{(2aN+b)}{N}\frac{1}{2aN+b}\sum_{k\leq 2aN+b}g(4a(x+c)-b^2+k^2)\\
&=&\left(2a+\frac{b}{N}\right)A_{2aN+b}^{x^2}g(4a(x+c)-b^2).
\end{eqnarray*}
Using this calculation and the previously established case of Theorem \ref{t:2} we obtain
\begin{eqnarray*}
\|A_{N}^{P}f\|_{\ell^{p'}(\mathbb{Z})}&\leq&\left(2a+\frac{b}{N}\right)\|A_{2aN+b}^{x^2}g\|_{\ell^{p'}(\mathbb{Z})}\\
&\leq& \left(2a+\frac{b}{N}\right)(2aN+b)^{2/p'-2/p}C_p\|g\|_{\ell^{p}(\mathbb{Z})}\\
&\leq& \left(2a+\frac{b}{N}\right)(2aN+b)^{2/p'-2/p}C_p\|f\|_{\ell^{p}(\mathbb{Z})},
\end{eqnarray*}
for every $3/2<p\leq2$.

\section{Reduction to Vinogradov's mean value theorem}
\begin{proof}[Proof of Theorem~\ref{thm:highdim}]
Let $\ell^2(\mathbb{Z}^d)\to\textup{L}^2(\mathbb{T}^d)$, $f\mapsto\widehat{f}$ denote the Fourier transform on the group $\mathbb{Z}^d$. For finitely supported $f$ this is simply the formation of the multiple Fourier series with coefficients $f(m)$; $m\in\mathbb{Z}^d$.

We  apply the Hausdorff--Young inequality twice to reduce the problem to bounds for the exponential sums. To do so, we write
\begin{align*}
& \widetilde{A}_N f(x_1,x_2,\ldots,x_d) \\
& = \frac{1}{N}\sum_{k=1}^{N} \int_{\mathbb{T}^d} \widehat{f}(t_1,t_2,\ldots,t_d) e^{2\pi i ( (x_1+k)t_1 + (x_2+k^2)t_2 + \cdots + (x_d+k^d) t_d) } dt_1 dt_2 \cdots dt_d \\
& = \int_{\mathbb{T}^d} \widehat{f}(t_1,t_2,\ldots,t_d) S_N(t_1,t_2,\ldots,t_d) e^{2\pi i ( x_1 t_1 + x_2 t_2 + \cdots + x_d t_d)} dt_1 dt_2 \cdots dt_d,
\end{align*}
where $S_N$ are the normalized exponential sums given by \eqref{eq:expsums}.
We recognize $\widetilde{A}_N f$ as the Fourier transform of the function $\widehat{f} \cdot  S_N$ on the group $\mathbb{T}^d$.
Applying the Hausdorff--Young inequality on $\mathbb{T}^d$, then H\"{o}lder's inequality on $\mathbb{T}^d$, and finally the Hausdorff--Young inequality on $\mathbb{Z}^d$, we get
\[ \|\widetilde{A}_N f\|_{\ell^{p'}(\mathbb{Z}^d)} \leq \|\widehat{f} \cdot S_N\|_{\textup{L}^{p}(\mathbb{T}^d)} \leq \|\widehat{f}\|_{\textup{L}^{p'}(\mathbb{T}^d)} \|S_N\|_{\textup{L}^{s}(\mathbb{T}^d)} \leq \|f\|_{\ell^{p}(\mathbb{Z}^d)} \|S_N\|_{\textup{L}^{s}(\mathbb{T}^d)}, \]
where $1/s=1/p-1/p'=2/p-1$.
Thus, the $\ell^p$-improving inequality depends on the $\textup{L}^s$-norm of the sums $S_N$.
Vinogradov's mean value theorem, as established by Bourgain, Demeter, and Guth \cite{MR3548534}, claims precisely the bound
\[ \|S_N\|_{\textup{L}^s(\mathbb{T}^d)} \lesssim_{d,s,\varepsilon} N^{-\frac{d(d+1)}{2s}+\varepsilon} \]
for any $s>d(d+1)$ and for any fixed $\varepsilon>0$. Note that in the typical formulation of Vinogradov's mean value theorem number $s$ needs to be an even integer, but the analytic proof from \cite{MR3548534} does not require that. Moreover, for $d\geq 3$ one can even remove the $\varepsilon$ by performing the Hardy--Littlewood circle method, as in \cite{MR3658168}*{Sec.~7} or \cite{MR3548534}*{Sec.~5}. Therefore, we actually have
\begin{equation}\label{eq:vinogradov}
\|S_N\|_{\textup{L}^s(\mathbb{T}^d)} \lesssim_{d,s} N^{-\frac{d(d+1)}{2s}}.
\end{equation}
Combining \eqref{eq:vinogradov} with the previous computation we get exactly the $\ell^p$-improving estimate \eqref{eq:highest} for $2-2/(d^2+d+1)<p\leq 2$ and $q=p'$.
Interpolation with the trivial estimates for $q=p$ settles the whole claimed range of $(p,q)$.
\end{proof}

\section{Projection of higher-dimensional averages to one-dimensional ones}
\begin{proof}[Proof of Theorem~\ref{thm:onedim}]
Take an arbitrary polynomial function $P\colon\mathbb{R}\to\mathbb{R}$ of degree $d\geq 3$ mapping $\mathbb{Z}$ back to $\mathbb{Z}$.
Let us write it as
\[ P(x) = a_0 + a_1 x + a_2 x^2 + \cdots + a_{d-1} x^{d-1} + a_d x^d, \]
where $a_0,a_1,\ldots,a_d\in\mathbb{R}$ and, without loss of generality, $a_d>0$.
By solving the Vandermonde linear system in the coefficients of $P$ the conditions $P(0),P(1),\ldots,P(d)\in\mathbb{N}$ imply that $a_0,a_1,\ldots,a_d$ are rational numbers. Moreover, let $v\in\mathbb{N}$ be the least common denominator of $a_1,\ldots,a_d$, so that we can write $a_j=b_j u/v$ for $j=1,\ldots,d$, where $u\in\mathbb{N}$ and $b_1,\ldots,b_d\in\mathbb{Z}$ do not have a common multiple greater than $1$.
By the formula for the Vandermonde determinant we also know that $v$ divides $\prod_{0\leq i<j\leq d}(j-i)$, so it has an upper bound depending only on $d$.
Finally, since the free coefficient of $P$ simply translates the averages \eqref{eq:polyaver}, it is safe to assume that $a_0=0$.

For any given function $g\colon\mathbb{Z}\to\mathbb{C}$ and a fixed number $r\in\{0,1,\ldots,u-1\}$ define $f\colon\mathbb{Z}^d\to \mathbb{C}$ by
\begin{align*}
f(x_1,x_2,\ldots,x_{d-1},x_d) & = \mathbbm{1}_{\{1,2,\ldots,2N\}}(x_1) \mathbbm{1}_{\{1,2,\ldots,2N^2\}}(x_2) \cdots \mathbbm{1}_{\{1,2,\ldots,2N^{d-1}\}}(x_{d-1}) \\
& \quad\ g(a_1 x_1 + a_2 x_2 + \cdots + a_{d-1} x_{d-1} + a_d x_d + r) \\
& \quad\ \mathbbm{1}_{\mathbb{Z}}(a_1 x_1 + a_2 x_2 + \cdots + a_{d-1} x_{d-1} + a_d x_d).
\end{align*}
This way
\begin{equation}\label{eq:choosecomb1}
\|f\|_{\ell^p(\mathbb{Z}^d)} \lesssim_{d,p} N^{\frac{(d-1)d}{2p}} \|g\|_{\ell^p(\mathbb{Z})}.
\end{equation}
For any $d$-tuple
\[ (x_1,x_2,\ldots,x_{d-1},x_d) \in \{1,2,\ldots,N\}\times\{1,2,\ldots,N^2\}\times\cdots\times\{1,2,\ldots,N^{d-1}\}\times\mathbb{Z} \]
satisfying $a_1 x_1 + a_2 x_2 + \cdots + a_d x_d\in \mathbb{Z}$ we have
\[ \widetilde{A}_N f(x_1,x_2,\ldots,x_d) = A_N g(a_1 x_1 + a_2 x_2 + \cdots + a_d x_d + r). \]
From the theory of linear Diophantine equations we know that $b_1\mathbb{Z} + b_2\mathbb{Z} + \cdots + b_d\mathbb{Z} = \mathbb{Z}$.
Hence, for sufficiently large $N\in\mathbb{N}$,
\begin{align}
\|\widetilde{A}_N f\|_{\ell^q (\mathbb{Z}^d)}^q
& \geq \sum_{\substack{x_1\in v\mathbb{Z}\\ 1\leq x_1\leq N}} \cdots
\sum_{\substack{x_{d-1}\in v\mathbb{Z}\\ 1\leq x_{d-1}\leq N^{d-1}}}
\sum_{x_d\in v\mathbb{Z}} \big|(A_N g)(a_1x_1+\cdots +a_d x_d+r)\big|^q \nonumber \\
& \gtrsim_P N^{\frac{(d-1)d}{2}} \sum_{n\in r+u\mathbb{Z}} |(A_N g)(n)|^q. \label{eq:choosecomb2}
\end{align}
Applying \eqref{eq:highest} and combining it with \eqref{eq:choosecomb1} and \eqref{eq:choosecomb2} we obtain
\[ N^{\frac{(d-1)d}{2q}} \Big(\sum_{n\in r+u\mathbb{Z}}|(A_N g)(n)|^q\Big)^{1/q} \lesssim_{P,p,q} N^{-\frac{d(d+1)}{2}(\frac{1}{p}-\frac{1}{q})} N^{\frac{(d-1)d}{2p}} \|g\|_{\ell^p(\mathbb{Z})}. \]
Finally, summing in $r=0,1,\ldots,u-1$ gives estimate \eqref{eq:polyest} for the function $g$.
\end{proof}

\section{Examples} \label{s:examples}
We formulate the examples that show certain sharpness in the $\ell^{p}$-improving inequalities.  These are essentially known, and we include them for completeness.

Counterexamples similar to the ones in \cite{2019arXiv190705734H} show that \eqref{eq:polyest} cannot hold outside the range
\begin{equation}\label{eq:polyconj}
\big\{ (p,q) \,:\, {\textstyle\frac{1}{q}\leq\frac{1}{p}},\ {\textstyle\frac{d}{q}\geq\frac{d-1}{p}},\ {\textstyle\frac{d-1}{q}\geq\frac{d}{p}-1} \big\}.
\end{equation}
Indeed, by taking $f=\mathbbm{1}_{\{P(1),P(2),\ldots,P(N)\}}$ we conclude
\[ \frac{\|A_N f\|_{\ell^q(\mathbb{Z})}}{\|f\|_{\ell^p(\mathbb{Z})}}
\geq \frac{|(A_N f)(0)|}{\|f\|_{\ell^p(\mathbb{Z})}} = \frac{1}{N^{1/p}}, \]
so that $d/q\geq(d-1)/p$. Similarly, by taking $f=\mathbbm{1}_{\{0\}}$ we get
\[ \frac{\|A_N f\|_{\ell^q(\mathbb{Z})}}{\|f\|_{\ell^p(\mathbb{Z})}}
\geq \frac{\big(\sum_{m=1}^{N}|(A_N f)(-P(m))|^q\big)^{1/q}}{\|f\|_{\ell^p(\mathbb{Z})}} = \frac{N^{1/q-1}}{1}, \]
which forces $(d-1)/q\geq d/p-1$.
If one only cares about the case $q=p'$, then \eqref{eq:polyconj} is simply the range
\begin{equation}\label{eq:polyrangeconj}
2-{\textstyle\frac{1}{d}}\leq p\leq 2.
\end{equation}
Theorem~\ref{thm:onedim} leaves a gap between the ranges \eqref{eq:polyrangespecial} and \eqref{eq:polyrangeconj} for large $d$.

Next, it is easy to see that \eqref{eq:highest} cannot hold outside the range
\[ \big\{ (p,q) \,:\, {\textstyle\frac{1}{q}\leq\frac{1}{p}},\ {\textstyle\frac{d^2+d}{q}\geq\frac{d^2+d-2}{p}},\ {\textstyle\frac{d^2+d-2}{q}\geq\frac{d^2+d}{p}-2} \big\}. \]
Indeed, by taking
\[ f = \mathbbm{1}_{\{(1^1,1^2,\ldots,1^d), (2^1,2^2,\ldots,2^d), \ldots, (N^1,N^2,\ldots,N^d)\}} \]
we conclude
\[ N^{-\frac{d(d+1)}{2}(\frac{1}{p}-\frac{1}{q})}
\gtrsim_{d,p,q} \frac{\|\widetilde{A}_N f\|_{\ell^q(\mathbb{Z}^d)}}{\|f\|_{\ell^p(\mathbb{Z}^d)}}
\geq \frac{|(\widetilde{A}_N f)(0,0,\ldots,0)|}{\|f\|_{\ell^p(\mathbb{Z}^d)}} = \frac{1}{N^{1/p}}, \]
so multiplying by $N^{1/p}$ and letting $N\to\infty$ give $(d^2+d)/q\geq(d^2+d-2)/p$.
Similarly, by taking $f=\mathbbm{1}_{\{(0,0,\ldots,0)\}}$ we get
\begin{align*}
N^{-\frac{d(d+1)}{2}(\frac{1}{p}-\frac{1}{q})}
&\gtrsim_{d,p,q} \frac{\|\widetilde{A}_N f\|_{\ell^q(\mathbb{Z}^d)}}{\|f\|_{\ell^p(\mathbb{Z}^d)}}
\\
& \geq \frac{\big( \sum_{m=1}^{N} |(\widetilde{A}_N f)(-m^1,-m^2,\ldots,-m^d)|^q \big)^{1/q}}{\|f\|_{\ell^p(\mathbb{Z})}}
= \frac{N^{1/q-1}}{1}.
\end{align*}
This forces $(d^2+d-2)/q\geq(d^2+d)/p-2$.
If one only cares about the cases $q=p'$, then we are talking about the range
\[ 2-{\textstyle\frac{2}{d^2+d}}\leq p\leq 2. \]
Comparing it to \eqref{eq:polyrangespecial} we see that Theorem~\ref{thm:highdim} is ``asymptotically optimal'' as $d\to\infty$.

\section{Conjectured sparse bounds }  \label{s:sparse}
Recall that a collection of intervals $ \mathcal S$ is said to be \emph{sparse} if for each interval $ I$ there is a subset $ E_I$ so that
$ \lvert  E_I\rvert  > \lvert  I\rvert/4   $ and
$ \{E_I : I\in \mathcal S\}$ are pairwise disjoint.
For $ 0\leq \lambda < 1$, and a sparse collection $ \mathcal S$, set
\begin{align*}
\Lambda _{p,q,\lambda } (f,g) &:= \sum_{I\in \mathcal S} \langle f \rangle _{p, I} \langle g \rangle _{I,q} \lvert  I\rvert ^{1- \lambda }  ,
\\
\textup{where} \quad \langle \phi  \rangle _{I,r} & := \Bigl[ \lvert  I\rvert ^{-1} \sum_{n \in I\cap \mathbb Z } \lvert  \phi (n)\rvert ^{r}   \Bigr] ^{1/r}.
\end{align*}

Concerning the maximal function $ A _{\ast } f = \sup_N A_N f $, where $ A_N$ is defined as in \eqref{eq:polyaver}, the main conjecture would
be that if $ (p,q) \in \mathcal A (P)$, that is when \eqref{eq:polyest} holds,  one has
\begin{equation}  \label{e:Asparse}
\langle  A _{\ast } f,g \rangle \lesssim \sup _{\mathcal S} \Lambda _{p,q',0} (f,g).
\end{equation}
Notice that we are using the conjugate index $ q' = q/(q-1)$ above.
In fact, the main result of \cite{2019arXiv190705734H} is that for the quadratic polynomial $ P (x) = x ^2 $, this is true, except possibly at the boundary of $ \mathcal A (P)$.  Nothing close to this is known for any other polynomial, as far as we know.

Turning to the fractional integral operator, the main conjecture would be that if $ (p,q, \lambda ) \in \mathcal I (P)$, that is when \eqref{e:Iest} holds,  one has
\begin{equation*}
\langle  I _{\lambda  } f,g \rangle \lesssim \sup _{\mathcal S} \Lambda _{p,q', \lambda } (f,g).
\end{equation*}
No such bound is known, even in the quadratic case.

The interest in sparse bounds comes in part as they immediately imply a range of weighted inequalities and vector valued inequalities.
The main results of \cites{2019arXiv190705734H,2019arXiv190902883H} concern sparse bounds for averages over the square integers and the primes. These results seem to be much more difficult than the improving or fractional integral inequalities.

\bibliographystyle{alpha,amsplain}

\end{document}